\documentclass[a4paper]{article}

\usepackage{amsthm,amsfonts,amssymb,amsmath,amsxtra}
\usepackage[all]{xy}
\usepackage[nohug,midshaft,heads=vee]{diagrams}
\diagramstyle[labelstyle=\scriptstyle]

\usepackage[colorlinks=true]{hyperref}

\usepackage[usenames]{color}

\usepackage{mathrsfs}

\theoremstyle{plain}
\newtheorem{thm}{Theorem}[section]
\newtheorem{lem}[thm]{Lemma}
\newtheorem{cor}[thm]{Corollary}
\newtheorem{prop}[thm]{Proposition}

\theoremstyle{definition}

\newtheorem{conj}[thm]{Conjecture}

\theoremstyle{remark}

    \newcommand{\FB}{{\mathbf{B}}}

    \newcommand{\sH}{{\mathscr{H}}}

    \newcommand{\BC}{{\mathbb {C}}}

    \newcommand{\BQ}{{\mathbb {Q}}} \newcommand{\BR}{{\mathbb {R}}}

     \newcommand{\BZ}{{\mathbb {Z}}}

    \newcommand{\CA}{{\mathcal {A}}} \newcommand{\CB}{{\mathcal {B}}}

   \newcommand{\F}{{\mathrm{F}}}

    \newcommand{\ad}{{\mathrm{ad}}}
    \newcommand{\Ad}{{\mathrm{Ad}}}

     \newcommand{\bsc}{{\mathrm{bsc}}}

    \newcommand{\Gal}{{\mathrm{Gal}}} 
    \newcommand{\GL}{{\mathrm{GL}}}
    \newcommand{\Hom}{{\mathrm{Hom}}}
    
    \newcommand{\id}{{\mathrm{id}}}
     \newcommand{\ind}{{\mathrm{ind}}}
    \newcommand{\Irr}{{\mathrm{Irr}}}

     \newcommand{\Out}{{\mathrm{Out}}}

    \newcommand{\red}{\mathrm{red}} \newcommand{\R}{{\mathrm{R}}}

    \newcommand{\res}{{\operatorname{res}}}

    \newcommand{\SL}{{\mathrm{SL}}}

    \newcommand{\op}{\mathrm{op}}

    \newcommand{\wt}{\widetilde}
    \newcommand{\wh}{\widehat}

    \newcommand{\incl}{\hookrightarrow}
    \newcommand{\sk}{\medskip}
    \newcommand{\lra}{\longrightarrow}
    \newcommand{\ra}{\rightarrow} 
    
    \newcommand{\bs}{\backslash}

    \newcommand{\s}{\sk\noindent}

\title{Distinction of regular depth-zero supercuspidal $L$-packets}
\author{Chong Zhang}
\begin{document}
\date{}
\maketitle

\begin{abstract}
In this paper, we study the relative Langlands program formulated by
Dipendra Prasad for Galois symmetric spaces. Under certain
assumptions, we confirm the necessary conditions in Prasad's
conjecture for regular depth-zero supercuspidal $L$-packets in the
sense of DeBacker-Reeder and Kaletha.
\end{abstract}

\section{Introduction}
\paragraph{Prasad's conjecture.} Let $F$ be a nonarchimedean
local field of characteristic 0. Let $G$ be a connected reductive
group defined over $F$, and $H$ a spherical subgroup of $G$. We say
that an irreducible admissible representation $\pi$ of $G(F)$ is
$H(F)$-distinguished if the space $\Hom_{H(F)}(\pi,1)$ is nonzero.
More generally, for a character $\omega$ of $H(F)$, we say that
$\pi$ is $\omega$-distinguished if $\Hom_{H(F)}(\pi,\omega)$ is
nonzero.  Under some assumptions on $G$ and the spherical variety
$X=G/H$, Sakellaridis and Venkatesh \cite{sv} formulate a relative
Langlands program whose aim is to understand the spectral
decomposition of $L^2\left(X(F)\right)$. Roughly speaking, they
conjecture that if $\pi$ is $H(F)$-distinguished then $\pi$ should
be a Langlands functorial lift from the  dual group $G_X$ attached
to $X$. When $X$ is a Galois symmetric space, Prasad \cite{pr1}
makes a more precise conjecture whose goal is to give sufficient and
necessary conditions, and even multiplicity formulas, for
representations to be distinguished.

Now let us introduce the setting of Prasad's conjecture. Let $E$ be
a quadratic field extension of $F$, $H$ a connected quasi-split
reductive group over $F$, and $G=\R_{E/F}H$ the Weil restriction of
$H$ with respect to the extension $E/F$. Twisting the group $H$ by
sending the nontrivial element $\sigma$ of $\Gal(E/F)$ to the
Chevalley involution $C$, we get a quasi-split reductive group
$H^\op$ over $F$. In other words, the group $H^\op$ is isomorphic to
$H$ over $E$, and its group of $F$-rational points is
$$H^\op(F)=\left\{g\in H(E):\ C(g)=\sigma(g)\right\}.$$ It is the group
$H^\op$ that serves the dual group $G_X$ in this framework. Let
$W_F$ be the Weil group of $F$, $\wh{G}$ the complex Langlands dual
group for $G$, and ${^LG}=\wh{G}\rtimes W_F$ the corresponding
$L$-group. The $L$-group $^L{H^\op}$ of $H^\op$ admits an
$L$-embedding $^LH^\op\ra{^LG}$. We refer the reader to
\S\ref{subsec. consequences} for more information about $H^\op$.

Let $H_\alpha$ be a pure inner twist of $H$ over $F$ such that
$\R_{E/F}H_\alpha=G$. Then the $F$-structure of $H_\alpha$ gives
rise to an action of $\sigma\in\Gal(E/F)$ on $G$ over $F$. We denote
this action by $\sigma_\alpha$, and if $H_\alpha=H$ we simply write
$\sigma$ instead of $\sigma_\alpha$. Note that, identifying
$H_\alpha(E)=H(E)=G(F)$, there exists $g_\alpha\in G(F)$ such that
$\sigma_\alpha=\Ad(g_\alpha)\circ\sigma$. For any reductive group
$A$ over $F$, Prasad considers a character $\omega_A$ of $A(F)$
associated with the quadratic extension $E$. The character
$\omega_A$ is trivial or quadratic. Let $\pi$ be an irreducible
admissible representation of $G(F)$. Prasad gives a conjectural
criterion of $\omega_{H_\alpha}$-distinction for $\pi$ in terms of
the Langlands parameter of $\pi$. Suppose that the conjectural local
Langlands correspondence holds for $G$. Then $\pi$ lies in an
$L$-packet $\Pi_\varphi(G)$, where $\varphi:W_F\ra{^LG}$ is a
Langlands parameter and $\Pi_\varphi(G)$ is a finite set of
equivalence classes of irreducible representations of $G(F)$
corresponding to the parameter $\varphi$. Let $\pi^{\sigma_\alpha}$
be the Galois conjugate of $\pi$ with respect to $\sigma_\alpha$.
Since $\sigma_\alpha$ and $\sigma$ differ by conjugation in $G(F)$,
we have $\pi^{\sigma_\alpha}\simeq\pi^\sigma$. Therefore the
equivalence class of $\pi^{\sigma_\alpha}$ does not depend on
$\alpha$. We denote by $\Pi_\varphi^\sigma(G)$ the $L$-packet that
$\pi^\sigma$ belongs to. On the other hand, we denote by $\pi^\vee$
the contragredient of $\pi$, and denote by $\Pi_\varphi^\vee(G)$ the
$L$-packet where $\pi^\vee$ lies in. As indicated by the notation,
we will explain in \S\S\ref{subsec. L-packets of
contragredient}--\ref{subsec. L-packets of galois conjugate} that
both $\Pi_\varphi^\vee(G)$ and $\Pi_\varphi^\sigma(G)$ do not depend
on $\pi$. In other words, we have
$$\Pi^\vee_\varphi(G)=\left\{\tau^\vee:\
\tau\in\Pi_\varphi(G)\right\}\quad\textrm{and}\quad
\Pi^\sigma_\varphi(G)=\left\{\tau^\sigma:\
\tau\in\Pi_\varphi(G)\right\}.$$ Part of \cite[Conjecture 2]{pr1}
can be stated as follows.

\begin{conj}\label{conj. prasad's main conjecture}
If $\pi$ is $\omega_{H_\alpha}$-distinguished, we
have\begin{enumerate}
\item $\Pi^\vee_{\varphi}(G)=\Pi^\sigma_{\varphi}(G)$.
\item The parameter $\varphi$ factors through $^LH^\op$.
\end{enumerate}
\end{conj}

Let $Z_\varphi$ be the centralizer of $\varphi$ in $\wh{G}$,
$C_\varphi$ the group of connected components of $Z_\varphi$, and
$\mu$ an irreducible representation of $C_\varphi$ corresponding to
$\pi$. Prasad \cite[Conjecture 2]{pr1} also provides a sufficient
condition for $\omega_{H_\alpha}$-distinction in terms of $\mu$, and
a conjectural formula for the multiplicity
$$\dim\Hom_{H_\alpha}(\pi,\omega_{H_\alpha})$$ in terms of
fibers of functorial lifts. Moreover, besides the nonarchimedean
case, the archimedean case $F=\BR$ is equally considered in
\cite{pr1}. See \cite[Conjecture 2]{pr1} for the precise statements.
Previous works by other authors for specific $H$ or $\pi$  are well
discussed in \cite{pr1}.

\paragraph{Main results.}
Assume that the characteristic of the residue field of $F$ is not 2.
In this paper, we study Conjecture \ref{conj. prasad's main
conjecture} in the case where $H$ is unramified, $E$ unramified over
$F$, and $\pi$ in a regular depth-zero supercuspidal $L$-packet.
Regular depth-zero supercuspidal $L$-packets are first constructed
by DeBacker and Reeder \cite{dr09} for pure inner twists of
unramified $p$-adic groups. Later Kaletha\cite{ka14} reinterprets
and enlarges their construction for extended pure inner twists.
These two kinds of $L$-packets are the same on pure inner twists,
and both satisfy the refined local Langlands conjectures (cf.
\cite[Conjectures E and F]{ka15}). We will adopt the formalism of
\cite{ka14}, and consider all the extended pure inner twists $H_a$
of $H$ over $F$. This formulation enables us to treat the problem of
distinction not only for quasi-split groups but also their extended
pure inner twists.

Now let $H$ be unramified. We take an extended pure inner twist
$H_a$ of $H$. Let $G_a=\R_{E/F}H_a$ and $\pi$ be an irreducible
admissible representation of $G_a(F)$. We always suppose that $\pi$
is in a regular depth-zero supercuspidal $L$-packet. The following
theorem (Theorem \ref{thm. theorem on representation}) is about the
representation $\pi$ itself rather than its $L$-packet.

\begin{thm}\label{thm. theorem on representation}
If $\pi$ is $\omega_{H_a}$-distinguished, we have
$$\pi^\vee\simeq\pi^\sigma.$$
\end{thm}

We mainly use Hakim-Murnaghan theory to prove Theorem \ref{thm.
theorem on representation}, and the proof heavily relies on the
construction of $\pi$. Now let $\varphi:W_F\ra{^LG}$ be a Langlands
parameter attached to $\pi$. To prove Conjecture \ref{conj. prasad's
main conjecture}, our strategy is to deduce the properties of
$\varphi$ from Theorem \ref{thm. theorem on representation}. The
regular depth-zero supercuspidal $L$-packet $\Pi_\varphi$ associated
with $\varphi$ is defined to be the disjoint union
$$\bigsqcup\Pi_\varphi(G_b)$$ where $G_b$ runs over extended pure
inner twists of $G$ over $F$, and $\Pi_\varphi(G_b)$ is the
$L$-packet of $G_b(F)$ attached to $\varphi$. Each $L$-packet
$\Pi_\varphi(G_b)$ is explicitly constructed. We will sometimes call
$\Pi_\varphi$ an {\em enlarged $L$-packet} to distinguish it from
$L$-packets $\Pi_\varphi(G_b)$. From the definition of $\Pi_\varphi$
we see that the word ``enlarged'' means that the $L$-packets of
extended pure inner twists are grouped together and thus can be
studied uniformly. Let $Z_\varphi$ be the centralizer of $\varphi$
in $\wh{G}$, which is an abelian group in our situation. There is a
nice parametrization of $\Pi_\varphi$ by the group $Z_\varphi^D$ of
characters of $Z_\varphi$. After fixing a Whittaker datum, the
correspondence
$$\iota: Z_\varphi^D\lra\Pi_\varphi$$ is canonical. We call $\iota$
an {\em enlarged local Langlands correspondence}. For $\mu\in
Z_\varphi^D$ such that $\iota(\mu)=\pi$ we call $(\varphi,\mu)$  a
{\em refined Langlands parameter} of $\pi$. We will study the
refined Langlands parameters of $\pi^\vee$ and $\pi^\sigma$ (see
Proposition \ref{prop. correspondence for galois conjugate}), and
obtain a relation between them (see Theorem \ref{thm. distinction
and parameters}) if $\pi$ is $\omega_{H_a}$-distinguished. With this
relation at hand, we can answer Conjecture \ref{conj. prasad's main
conjecture}. We say that an extended pure inner twist $G_b$ of $G$
comes from $H$ if $G_b=\R_{E/F}H_b$ for some extended pure inner
twist $H_b$ of $H$. We consider a subset $\Pi_\varphi^\circ$ of
$\Pi_\varphi$ which is a disjoint union of $\Pi_\varphi(G_b)$ for
those $G_b$ coming from $H$. As before it makes sense to define the
Galois conjugate $\tau^\sigma$ for $\tau\in\Pi_\varphi^\circ$. Set
$$\Pi^{\circ,\vee}_\varphi=\left\{\tau^\vee:\
\tau\in\Pi^\circ_\varphi\right\}\quad\textrm{and}\quad
\Pi^{\circ,\sigma}_\varphi=\left\{\tau^\sigma:\
\tau\in\Pi^\circ_\varphi\right\}.$$ The advantage of the enlarged
local Langlands correspondence is that it enables us to compare
$\Pi^{\circ,\vee}_\varphi$ with $\Pi^{\circ,\sigma}_\varphi$, and
not merely $\Pi_\varphi^\vee(G_a)$ with $\Pi_\varphi^\sigma(G_a)$.

\begin{thm}\label{thm. intro. distinction and packets}
If $\pi$ is $\omega_{H_a}$-distinguished, we  have
\begin{enumerate}
\item $\Pi^{\circ,\vee}_{\varphi}=\Pi^{\circ,\sigma}_{\varphi}$.
\item The parameter $\varphi$ factors through $^LH^\op$.
\end{enumerate}
\end{thm}

In fact, the sets $\Pi_\varphi^{\circ,\vee}$ and
$\Pi_\varphi^{\circ,\sigma}$ are subsets of certain enlarged
$L$-packets. The $L$-group $^LG$ is equipped with two involutions
$C$ and $\delta$, which are related to the actions of taking
contragredient and Galois conjugate on the representations
respectively. Thus, starting from the parameter $\varphi$, we have
two other parameters $C\circ\varphi$ and $\delta\circ\varphi$.
Kaletha \cite[\S5]{ka13} shows that
$\Pi^{\vee}_\varphi(G_b)=\Pi_{C\circ\varphi}(G_b)$ for any extended
pure inner twist $G_b$ of $G$. Therefore $\Pi_\varphi^{\circ,\vee}$
is a subset of $\Pi_{C\circ\varphi}$. On the other hand, we will
show that $\Pi^\sigma_\varphi(G_b)=\Pi_{\delta\circ\varphi}(G_b)$
for any extended pure inner twist $G_b$ coming from $H$ (see
Proposition \ref{prop. packet of galois conjugate}). Thus
$\Pi_\varphi^{\circ,\sigma}$ is a subset of
$\Pi_{\delta\circ\varphi}$. Hence Theorem \ref{thm. intro.
distinction and packets 2} below (also see Corollary \ref{cor.
distinction and packets}) implies Theorem \ref{thm. intro.
distinction and packets}, and offers positive evidence to Conjecture
\ref{conj. prasad's main conjecture}. The proof of Theorem \ref{thm.
intro. distinction and packets 2} is a direct consequence of Theorem
\ref{thm. distinction and parameters} and the description of
$^LH^\op$.

\begin{thm}\label{thm. intro. distinction and packets 2}
If $\pi$ is $\omega_{H_a}$-distinguished, we have
\begin{enumerate}
\item $\Pi_{C\circ\varphi}=\Pi_{\delta\circ\varphi}$.
\item The parameter $\varphi$ factors through $^LH^\op$.
\end{enumerate}
\end{thm}

\paragraph{Organization of this article.} Conjecture \ref{conj. prasad's main
conjecture} only involves representations of quasi-split groups,
while our treatment deals with more general reductive groups. We
recall the local Langlands conjectures for quasi-split groups in
\S\ref{sec. notation}, and the enlarged local Langlands
correspondence for extended pure inner twists in \S\ref{subsec.
construction of L-packets} where we only focus on regular depth-zero
supercuspidal representations. In the subsequent sections
\S\S\ref{subsec. L-packets of contragredient}--\ref{subsec.
L-packets of galois conjugate}, we show the relationship between
actions (i.e. contragredient and Galois conjugate) on the
representations with actions (i.e. involutions $C$ and $\delta$) on
the parameters. Main theorems of the paper and their proof are given
in \S\ref{sec. distinction}. The proof of Theorem \ref{thm. theorem
on representation} relies on Proposition \ref{prop. stable torus}
whose proof involves Hakim-Murnaghan theory for distinguished tame
supercuspidal representations and the theory for distinguished
representations over finite fields. Theorems \ref{thm. intro.
distinction and packets} and \ref{thm. intro. distinction and
packets 2} follow from Theorem \ref{thm. distinction and parameters}
whose proof relies on Theorem \ref{thm. theorem on representation}
and Proposition \ref{prop. correspondence for galois conjugate}.

\paragraph{Notation and conventions.}
Let $F$ be a $p$-adic field (i.e. a finite extension of $\BQ_p$)
with ring of integers $O_F$ and residue field $k_F$. For the
construction of regular supercuspidal $L$-packets, we assume that
$p$ is odd. Let $E$ be an unramified quadratic field extension of
$F$ with ring of integers $O_E$ and residue field field $k_E$. We
write $W_F$ and $W_E$ for the Weil groups of $F$ and $E$
respectively, and write $I_F$ for the inertia subgroup of $W_F$.
Denote by $\sigma$ the nontrivial automorphism in $\Gal(E/F)$. Fix
an algebraic closure $\bar{F}$ such that $E\subset\bar{F}$, and
denote by $\Gamma$ the absolute Galois group $\Gal(\bar{F}/F)$. Let
$F^u$ be the maximal unramified extension of $F$ in $\bar{F}$.

If $H$ is an algebraic group over $F$, we use $\R_{E/F}H$ to denote
its Weil restriction attached to the extension $E/F$. Thus
$\R_{E/F}H$ is an algebraic group over $F$ whose group of
$F$-rational points is $H(E)$. The $F$-rational structure of $H$
gives rise to an action of $\Gal(E/F)$ on $\R_{E/F}H$. Denote by
$\theta$ the involution on $\R_{E/F}H$ induced by $\sigma$. If $W$
is a subgroup of $\R_{E/F}H$, by abuse of notation, we denote the
image of $W$ under $\theta$ by $W^\sigma$, and denote by $W^\theta$
the subgroup of fixed points. If $(\pi,V_\pi)$ is a representation
of $H(E)$ where $V_\pi$ is the underlying space of $\pi$, the Galois
conjugate $\pi^\sigma$ of $\pi$ is a representation of $H(E)$ with
underlying space $V_\pi$, defined by $\pi^\sigma(g)v=\pi(g^\sigma)v$
for $v\in V_\pi$. We will use similar notation when we discuss
objects over finite fields.

For a connected reductive group $G$ over $F$, we denote by
$\CB^\red(G,F)$ the reduced Bruhat-Tits building of $G(F)$. For any
$x\in\CB^\red(G,F)$, we write $G(F)_{x,0}$ for the parahoric
subgroup corresponding to $x$, $G(F)_{x,0+}$ for its pro-unipotent
radical, and $\mathsf{G}$ for the corresponding connected reductive
group over $k_F$. For any unramified maximal torus $S$ of $G$, we
denote the intersection $\CA^\red(S,F^u)\cap\CB^\red(G,F)$ by
$\CA^\red(S,F)$, where $\CA^\red(S,F^u)$ is the reduced apartment of
$S$ in $\CB^\red(G,F^u)$.

If $A$ is a group, we use $A^D$ to denote $\Hom(A,\BC^\times)$, and
use $\Irr(A)$ to denote the set of equivalence classes of
irreducible representations of $A$. If $A$ is a topological group,
we use $\pi_0(A)$ to denote its group of connected components.

\paragraph{Acknowledgements.}
This work was partially supported by NSFC 11501033. The author
thanks Dipendra Prasad and Wen-Wei Li for their valuable comments on
earlier versions of this paper. He also thanks the anonymous
referees for the careful reading and  helpful suggestions to improve
the paper.

\section{$L$-packets}\label{sec. L-packets}
\subsection{Local Langlands conjectures}\label{sec.
notation}

There are several versions of the local Langlands conjectures, that
is, the version for quasi-split groups, the enlarged version for
pure inner twists of quasi-split groups, and the recent version for
extended pure inner twists and rigid inner twists of quasi-split
groups. We refer to \cite{ka15} for an excellent survey on the
relation between these versions of conjectures. In this subsection,
we briefly review the local Langlands conjecture for quasi-split
groups.

Let $G$ be a connected reductive group (not necessary be
quasi-split) defined over $F$, $\wh{G}$ the complex Langlands dual
group of $G$, and $^LG=\wh{G}\rtimes W_F$ the Weil-form $L$-group of
$G$. The Langlands parameters (considered in this paper) are
continuous homomorphisms
$$\varphi:W_F\ra{^LG}$$ which are sections of the natural projection
${^LG}\ra W_F$.  For each $L$-parameter $\varphi$ up to
$\wh{G}$-conjugate, the basic form of the local Langlands conjecture
predicts the existence of a finite set $\Pi_\varphi(G)$, called an
$L$-packet, of equivalence classes of irreducible admissible
representations of $G(F)$. If two parameters $\varphi$ and
$\varphi'$ are not $\wh{G}$-conjugate, the packets
$\Pi_{\varphi}(G)$ and $\Pi_{\varphi'}(G)$ should be disjoint. Each
irreducible admissible representation $\pi$ of $G(F)$ should belong
to a unique packet.

Now we suppose further that $G$ is quasi-split. The refined local
Langlands conjecture says that we can parameterize representations
in $\Pi_\varphi(G)$ in terms of the information of the parameter
$\varphi$. To be more precise, let $Z_\varphi$ be the centralizer of
$\varphi$ in $\wh{G}$, and $Z(\wh{G})$ the center of $\wh{G}$. Then
there should exist a bijective map
\begin{equation}\label{equ. correspondence for quasi-split}
\iota:\Irr\left(\pi_0\left(Z_\varphi/Z(\wh{G})^\Gamma\right)\right)
\lra\Pi_\varphi(G),\end{equation} and this map is unique in the
following sense. Recall that a Whittaker datum for $G$ is a
$G(F)$-conjugacy class of pairs $(B,\psi)$, where $B$ is a Borel
subgroup of $G$ defined over $F$ with unipotent radical $U$, and
$\psi$ is a non-degenerate character $U(F)\ra\BC^\times$. Given a
Whittaker datum $(B,\psi)$, an admissible representation $\pi$ is
called $(B,\psi)$-generic if $\Hom_{U(F)}(\pi,\psi)\neq0$. Fix a
Whittaker datum $(B,\psi)$. The map $\iota$ should be unique in the
sense that it sends the trivial representation in
$\Irr\left(\pi_0\left(Z_\varphi/Z(\wh{G})^\Gamma\right)\right)$ to
the conjecturally unique $(B,\psi)$-generic representation in
$\Pi_\varphi(G)$, and satisfies the endoscopic character identities
(which we will not review). We denote by $\iota_{B,\psi}$ this
conjectural map.

\subsection{Regular depth-zero supercuspidal $L$-packets}
\label{subsec. construction of L-packets}
\paragraph{Extended pure inner twists.}
To parameterize representations of non-quasi-split groups, one has
to consider all inner twists of the quasi-split group $G$, whose
isomorphism classes are parameterized by $H^1(F,G_\ad)$ where
$G_\ad$ is the adjoint group of $G$. However, inner twists are not
rigid enough to ensure the existence of a bijection as (\ref{equ.
correspondence for quasi-split}), since the automorphism group of an
inner twist is bigger than its inner automorphism group. Vogan
introduces the notion of pure inner twists to rigidify inner twists.
The remaining problem is that not all of the inner twists can be
rigidified to be pure. To include more inner twists which can be
rigidified, based on his work on isocrystals with additional
structure, Kottwitz introduces the notion of extended pure inner
twists. We adopt the formulation of extended pure inner twists in
\cite[\S2]{ka14} and refer to \cite{kot} for a complete
introduction.

Let $L$ be the completion of $F^u$, and $\bar{L}$ a fixed algebraic
closure of $L$ such that $\bar{F}\subset\bar{L}$. There exists a
subset $Z^1(W_F,G(\bar{L}))_\bsc$ of {\em basic} 1-cocycles in
$Z^1(W_F,G(\bar{L}))$. The set of cohomology classes of
$Z^1(W_F,G(\bar{L}))_\bsc$ is denoted by $\FB(G)_\bsc$. The
definition of $Z^1(W_F,G(\bar{L}))_\bsc$ implies that there are
natural maps
$$Z^1(W_F,G(\bar{L}))_\bsc\lra Z^1(F,G_\ad),
\quad\textrm{and}\quad \FB(G)_\bsc\lra H^1(F,G_\ad).$$ An {\em
extended pure inner twist} of $G$ is a pair $(\xi,z)$ where
$\xi:G\ra G'$ is an inner twist and $z$ is in
$Z^1(W_F,G(\bar{L}))_\bsc$ such that the image of $z$ in
$Z^1(F,G_\ad)$ is the cocycle $\gamma\mapsto \xi^{-1}\gamma(\xi)$.
The map $(\xi,z)\mapsto z$ establishes a bijection between the set
of isomorphism classes of pure inner twists of $G$ and the set
$\FB(G)_\bsc$.

\paragraph{Regular depth-zero supercuspidal $L$-packet data.}
From now on, let $G$ be a connected unramified reductive group over
$F$. Let $\varphi:W_F\ra{^LG}$ be a Langlands parameter and
$\varphi_0$ its projection to $\wh{G}$. Recall that $\varphi$ is
called a {\em TRSELP} if the restriction of $\varphi_0$ to the wild
inertia subgroup of $W_F$ is trivial, the centralizer of
$\varphi_0(I_F)$ in $\wh{G}$ is a maximal torus, and the index of
$Z(\wh{G})^\Gamma$ in $Z_\varphi$ is finite. The notion TRSELP is
the abbreviation of ``tame regular semisimple elliptic Langlands
parameter'', which was first introduced in \cite[page 825]{dr09}.

Now let $\varphi$ be a TRSELP. Choose a hyperspecial vertex $o$ in
$\mathcal{B}^\red(G,F)$. The vertex $o$ determines an
$O_F$-structure on $G$. A {\em regular depth-zero supercuspidal
$L$-packet datum} associated with $\varphi$ is a quadruple
$$\left(S,{^Lj},a,\chi\right)$$ such that
\begin{itemize}
\item $S$ is an unramified elliptic maximal torus of $G$ defined over $O_F$,
\item ${^Lj}$ is an unramified $L$-embedding ${^LS\ra{^LG}}$,
\item $a$ is a Langlands parameter $W_F\ra {^LS}$
satisfying
$${^Lj}\circ a=\varphi,$$
\item $\chi: S(F)\lra\BC^\times$ is the character attached to the
parameter $a$ by the local Langlands correspondence for
tori.\end{itemize} The conditions on $\varphi$ implies that $\chi$
is regular. The choices of regular depth-zero supercuspidal
$L$-packet data depend on the vertex $o$, and are unique up to
$G(O_F)$-conjugate (\cite[Lemma 3.4.1]{ka14}) after fixing a vertex.

\paragraph{Regular depth-zero supercuspidal $L$-packets.}
Let $\varphi$ be a TRSELP, and $\left(S,{^Lj},a,\chi\right)$ a
regular depth-zero supercuspidal $L$-packet datum attached to
$\varphi$. Now we review the construction of the enlarged $L$-packet
$\Pi_{\varphi}$. For each $b\in\FB(G)_\bsc$, denote by $G_b$ a
representative in the isomorphism class of the extended pure inner
twists corresponding to $b$. The packet $\Pi_\varphi$ is defined to
be a disjoint union of the $L$-packets $\Pi_\varphi(G_b)$ where $b$
runs over $\FB(G)_\bsc$ and $\Pi_{\varphi}(G_b)$ consists of certain
depth-zero supercuspidal representations of $G_b(F)$. The
construction of $\Pi_\varphi(G_b)$ is as follows.

First let us recall how every element $\mu$ of $Z_{\varphi}^D$ gives
rise to a cocycle $z_\mu$ in $Z^1(W_F,G(\bar{L}))_\bsc$  (cf.
\cite[\S3.3]{ka14}). The map ${^Lj}$ induces a bijection
\begin{equation}\label{equ. bijection cocharacters}
Z^D_{\varphi}\lra X_*(S)_\Gamma,\quad
\mu\mapsto\bar{\lambda}_\mu.\end{equation} Choose a lift
$\lambda_\mu$ of $\bar{\lambda}_\mu$ in $X_*(S)$ arbitrarily.  The
choice of $\lambda_\mu$ does not matter (see \cite[Lemma
3.3.2]{ka14}). The assignment $\Phi\mapsto\lambda_\mu(\varpi)$
extends to a 1-cocycle $W_F\ra S(\bar{L})$ and its prolongation to
$G(\bar{L})$ is $z_\mu$, where $\Phi\in W_F$ is the inverse of a
fixed Frobenius automorphism and $\varpi\in O_F$ is a fixed
uniformizer. Let $\xi_\mu:G\ra G_\mu$ be the inner twist determined
by $z_\mu$. Denote by
$$\mathsf{b}: Z_{\varphi}^D\lra\FB(G)_\bsc$$ the map induced by $\mu\mapsto
z_\mu$. Actually the map $\mathsf{b}$ is the composition of the
natural map $Z_\varphi^D\ra\left(Z(\wh{G})^\Gamma\right)^D$ and the
isomorphism
$$\left(Z(\wh{G})^\Gamma\right)^D\stackrel{\sim}{\lra}
\FB(G)_\bsc.$$

Next for each $\mu\in Z_{\varphi}^D$, consider the embedding of $S$
into $G_\mu$
$$j_\mu:S\lra G\stackrel{\xi_\mu}{\lra} G_\mu,$$ which is in fact an
admissible embedding defined over $F$. Denote by $S_\mu$ the image
of $S$ in $G_\mu$. Then $S_\mu$ is an elliptic unramified maximal
torus of $G_\mu$, and isomorphic to $S$ over $F$. Set
$\chi_\mu=\chi\circ j_\mu^{-1}$, which is a character of $S_\mu(F)$.
Based on $(S_\mu,\chi_\mu)$, we can construct an irreducible
depth-zero supercuspidal representation
$\pi\left(S_\mu,\chi_\mu\right)$ of $G_\mu(F)$. The condition on
$\varphi$ implies that $\chi$ and thus $\chi_\mu$ are of depth zero
and regular. Let $x=\CA^\red(S_\mu,F)$, which is a vertex in
$\CB^\red(G_\mu,F)$ since $S_\mu$ is unramified and elliptic (cf.
\cite[\S2.2]{de06}). Let $\mathsf{G}_\mu$ be the corresponding
connected reductive group over $k_F$ associated to $x$, and
$\mathsf{S}_\mu$ the maximal torus in $\mathsf{G}_\mu$ corresponding
to $S_\mu$. The character $\chi_\mu$ gives rise to a regular
character $\overline{\chi_\mu}:\mathsf{S}_\mu(k_F)\ra\BC^\times$,
and thus an irreducible cuspidal representation
$\overline{\rho_\mu}$ of $\mathsf{G}_\mu(k_F)$ by Deligne-Lusztig
theory. Denote by $\rho_\mu$ the inflation of $\overline{\rho_\mu}$
to $G_\mu(F)_{x,0}$. Set
$$\pi(S_\mu,\chi_\mu)=\ind^{G_\mu(F)}_{Z(F)G_\mu(F)_{x,0}}(\rho_\mu\otimes\chi_\mu),$$
where $Z$ is the center of $G_\mu$, and the induction is the compact
induction.

Lastly, set
\begin{equation}\label{equ. definition of single L-packet}\Pi_{\varphi}(G_b)
=\left\{\pi\left(S_\mu,\chi_\mu\right):\ \mu\in Z_{\varphi}^D\
\textrm{such that} \ \mathsf{b}(\mu)=b\right\},\end{equation} and
set
\begin{equation}\label{equ. definition of L-packet}
\Pi_{\varphi}=\bigsqcup\limits_{b\in\FB(G)_\bsc}\Pi_{\varphi}(G_b).
\end{equation}

\paragraph{Enlarged local Langlands correspondence.}
According to the definition of $\Pi_{\varphi}$ (\ref{equ. definition
of L-packet}), we have a bijection
$$\iota: Z_{\varphi}^D\lra\Pi_{\varphi},\quad
\mu\mapsto\pi(S_\mu,\chi_\mu),$$ which depends only on the
$G(F)$-orbit of the chosen hyperspecial vertex $o$. Moreover the
image $\pi(S,\chi)$ of the trivial character under $\iota$ is
generic. More precisely, let $(B,\psi)$ be a Whittaker datum which
satisfies that the reduced apartment of the maximal torus of $B$
contains $o$, and $\psi:U(F)\ra\BC^\times$ reduces to a generic
character of $U(k_F)$.  Then $\pi(S,\chi)$ is $(B,\psi)$-generic
\cite[Lemma 6.2.1]{dr09}. We will denote $\iota$ by $\iota_{B,\psi}$
to indicate the base point. As \cite[Theorem 4.3.3]{ka14} shows, the
parametrization $\iota_{B,\psi}$ makes the $L$-packet $\Pi_\varphi$
satisfying the endoscopic character identities. In summary we have
the following commutative diagram
$$\begin{diagram}
Z_{\varphi}^D&&\rTo^{\iota_{B,\psi}}&\bigsqcup_{b \in\FB(G)_\bsc}
\Pi_{\varphi}(G_b)\\
\dTo&&&\dTo\\
\left(Z(\wh{G})^\Gamma\right)^D&&\rTo&\FB(G)_\bsc\\
\end{diagram}$$

\subsection{$L$-packets of the contragredient}\label{subsec. L-packets of
contragredient}

For a general connected reductive group $G$ over $F$ and a Langlands
parameter $\varphi$ for $G$, Adams and Vogan \cite{av16} conjecture
that the contragredient $\Pi_\varphi^\vee(G)$ of its $L$-packet
$\Pi_{\varphi}(G)$, defined by
$$\Pi^\vee_{\varphi}(G)=\left\{\pi^\vee:\ \pi\in\Pi_{\varphi}(G)\right\},$$
is also an $L$-packet whose Langlands parameter is the composition
of $\varphi$ with the Chevalley involution of $^LG$. Moreover,
Prasad \cite{pr1} and Kaletha \cite{ka13} conjecture that the
refined Langlands parameter for each representation $\pi^\vee$ in
$\Pi_\varphi^\vee(G)$ can be explicitly determined. When $G$ is
unramified and $\varphi$ is a TRSELP, we briefly review the enlarged
local Langlands correspondence for $\Pi^\vee_\varphi$ (defined in
the same way as $\Pi^\vee_\varphi(G)$), which is shown in
\cite[\S5]{ka13}.

\paragraph{Chevalley involution.} Let $G$ be quasi-split. To define a Chevalley
involution, fix an $F$-splitting $(T,B,\{X_\alpha\})$ (see
\cite[\S1.3]{kot84} for the definition of splitting). The Chevelley
involution $C$ attached to this splitting is the unique involution
on $G$ defined over $F$ such that the restriction of $C$ to the
maximal torus $T$ is the inverse map, $C(B)=B^\op$ where $B^\op$ is
the opposite of the Borel group $B$, and $C(X_\alpha)=X_{-\alpha}$.
Chevalley involutions attached to different $F$-splittings are all
$G(\bar{F})$-conjugate. We refer to \cite[\S2]{av16} for more
information about the Chevalley involution.

Similarly, fixing an $F$-splitting
$(\wh{T},\wh{B},\{X_{\wh{\alpha}}\})$ for the complex dual group
$\wh{G}$, we have a unique Chevelley involution $\wh{C}$ which
commutes with the action of $\Gamma$ on $\wh{G}$. Thus we obtain an
$L$-automorphism $^LC=\wh{C}\times\id_{W_F}$ of $^LG$. For
simplicity, we will use $C$ to denote $\wh{C}$ or $^LC$ if there is
no confusion.

\paragraph{Langlands parameter.}
Now suppose that $G$ is unramified and $\varphi$ is a TRSELP for
$G$. Let $\left(S,{^Lj},a,\chi\right)$ be a regular depth-zero
supercuspidal $L$-packet datum for $\varphi$. Choose an
$F$-splitting $(\wh{T},\wh{B},\{X_{\wh{\alpha}}\})$ for $\wh{G}$ so
that $\wh{T}={^Lj(\wh{S})}$, and let $C$ be the Chevalley involution
on $^LG$ with respect to this splitting. Then the contragredient
$\Pi^\vee_\varphi$ of the enlarged $L$-packet $\Pi_\varphi$ is the
enlarged $L$-packet associated with the Langlands parameter
$C\circ\varphi$. In other words, we have
$$\Pi^\vee_{\varphi}=\Pi_{C\circ\varphi}.$$
The reason is as follows. The parameter $C\circ\varphi$ is a TRSELP,
whose regular depth-zero supercuspidal $L$-packet datum can be
chosen to be $\left(S,{^Lj},a^{-1},\chi^{-1}\right)$ where $a^{-1}$
is the Langlands parameter corresponding to $\chi^{-1}$. Note that
$Z_{\varphi}=Z_{C\circ\varphi}$. For each $\mu\in Z_{\varphi}^D$,
which is also viewed as an element in $Z_{C\circ\varphi}^D$, we have
\begin{equation}\label{equ. induction and contragredient}
\pi\left(S_\mu,\chi_\mu\right)^\vee=\pi\left(S_\mu,\chi_\mu^{-1}\right)
=\pi\left(S_\mu,(\chi^{-1})_\mu\right).\end{equation} Also note
that, since $\pi\left(S,\chi\right)$ is $(B,\psi)$-generic,
$\pi\left(S,\chi\right)^\vee$ is $(B,\psi^{-1})$-generic. Therefore,
the relation between the refined local Langlands correspondence for
$$\iota_{B,\psi^{-1}}: Z_{C\circ\varphi}^D\lra\Pi_{C\circ\varphi}$$ and
$$\iota_{B,\psi}: Z_{\varphi}^D\lra\Pi_{\varphi}$$ is
\begin{equation}\label{equ. correspondence for contragredient}
\iota_{B,\psi^{-1}}(C\circ\varphi,\mu)
=\iota_{B,\psi}(\varphi,\mu)^\vee.\end{equation}

\subsection{$L$-packets of the Galois conjugate}\label{subsec. L-packets of galois
conjugate} From now on, until the end of the paper, let $H$ be an
unramified connected reductive group over $F$, $E$ an unramified
quadratic field extension of $F$, and $G=\R_{E/F}H$ which is also
unramified over $F$.

\paragraph{Galois conjugate.}
The restriction map
$$ H^1\left(W_F,
H(\bar{L})\right){\lra} H^1\left(W_E,H(\bar{L})\right)$$ gives rise
to
$$\res: \FB(H)_\bsc{\lra}
\FB(G)_\bsc.$$  Denote by $\FB(G)_\bsc^\circ$ the image of
$\FB(H)_\bsc$ under the restriction map. Note that $b$ lies in
$\FB(G)_\bsc^\circ$ if and only if there exists $a\in\FB(H)_\bsc$
such that $G_b=\R_{E/F}H_a$, and the inner twist $\xi_b:G\ra G_b$ is
induced from the inner twist $\xi_a:H\ra H_a$. The reason that we
introduce the notion $\FB(G)_\bsc^\circ$ is that we want to consider
the Galois action of $\Gal(E/F)$ on $G_b$. For
$b\in\FB(G)_\bsc^\circ$ and any $a\in\FB(H)_\bsc$ such that
$\res(a)=b$, the $F$-structure of $H_a$ induces an action, denoted
by $\sigma_a$, of $\sigma$ on $G_b$. Moreover, if there are two
elements $a$ and $a'$ of $\FB(H)_\bsc$ such that
$\res(a)=\res(a')=b$, then there exists $g\in G_b(F)$ such that
$\sigma_a=\Ad(g)\circ\sigma_{a'}$. Therefore for any representation
$\pi$ of $G_b(F)$ we have $\pi^{\sigma_a}\simeq\pi^{\sigma_{a'}}$.
Hence the notion $\pi^\sigma$ is well-defined for equivalence
classes of representations of $G_b(F)$ where
$b\in\FB(G)^\circ_\bsc$.

Let $\varphi:W_F\ra{^LG}$ be a TRSELP. We consider a subset
$Z_{\varphi}^{D,\circ}$ of $Z_{\varphi}^D$, which is defined to be
$$Z_\varphi^{D,\circ}=\left\{\mu\in Z_\varphi^D:\
\mathsf{b}(\mu)\in\FB(G)_\bsc^\circ\right\}.$$  Set
$$\Pi^\circ_{\varphi}=\bigsqcup\limits_{b\in\FB(G)^\circ_\bsc}
\Pi_{\varphi}(G_b).$$ For each $b\in\FB(G)_\bsc^\circ$, define
$$\Pi_\varphi^\sigma(G_b)=\left\{\pi^\sigma:\
\pi\in\Pi_{\varphi}(G_b)\right\}.$$ Set
$$\Pi^{\circ,\sigma}_{\varphi}=\bigsqcup_{b\in\FB(G)^\circ_\bsc}
\Pi^{\sigma}_{\varphi}(G_b).$$

\paragraph{Langlands parameter.}
We identify $^LG$ with $\left(\wh{H}\times\wh{H}\right)\rtimes W_F$.
The involution $\theta$ on $G$ induces an $L$-automorphism $\delta$
on ${^LG}$:
$$\delta(x,y,w)=(y,x,w),\quad \textrm{for}\ x,y\in\wh{H},\ w\in W_F.$$ Then the
composition $\delta\circ\varphi$ is also a TRSELP. Define a subset
$Z_{\delta\circ\varphi}^{D,\circ}$ of $Z^D_{\delta\circ\varphi}$ to
be $$\left\{\mu\in Z_{\delta\circ\varphi}^D:\
\mathsf{b}(\mu)\in\FB(G)_\bsc^\circ\right\},$$ and set
$$\Pi^\circ_{\delta\circ\varphi}=\bigsqcup\limits_{b\in\FB(G)^\circ_\bsc}
\Pi_{\delta\circ\varphi}(G_b).$$ The rest of this section is devoted
to proving the following proposition as well as a more precise
statement (Proposition \ref{prop. correspondence for galois
conjugate}), which basically shows that
$\Pi_\varphi^{\circ,\sigma}(G_b)$ for $b\in\FB(G)^\circ_\bsc$ is an
$L$-packet whose Langlands parameter is $\delta\circ\varphi$. The
arguments are analogous to those in \cite[\S5]{ka13}, which we have
recalled in \S\ref{subsec. L-packets of contragredient}.

\begin{prop}\label{prop. packet of galois conjugate}
We have
$$\Pi_{\varphi}^{\circ,\sigma}=\Pi_{\delta\circ\varphi}^\circ.$$
\end{prop}
\begin{proof}
Recall that, to construct the bijection $\iota_{B,\psi}:
Z_\varphi^D\ra\Pi_\varphi$, we need to fix a hyperspecial vertex
$o\in\CB^\red(G,F)$ and a Whittaker datum $(B,\psi)$. Now we require
that $o$ and $B$ (and thus $U$) are $\sigma$-stable.

First we choose a regular depth-zero supercuspidal $L$-packet datum
$\left(S,{^Lj},a,\chi\right)$ for $\varphi$. Then $S^\sigma$ is also
an unramified elliptic maximal torus of $G$ defined over $O_F$, and
$\chi^\sigma=\chi\circ\sigma$ is a character of $S^\sigma(F)$. Let
$a^\sigma: W_F\ra {^L(S^\sigma)}$ be the Langlands parameter
associated with $\chi^\sigma$. Denote by $\left({^LS}\right)^\delta$
the image of ${^LS}$ under $\delta$ in ${^LG}$. According to the
construction of $L$-groups with respect to restriction of scalars,
we can identify ${^L(S^\sigma)}$ with $\left({^LS}\right)^\delta$,
and we have
$$a^\sigma=\delta\circ a.$$ Therefore, we have the following
commutative diagram \begin{diagram}
{^L(S^\sigma)}&&\rTo^{{^Lj^\delta}}&{^LG}\\
\uTo^\delta&&&\uTo^\delta\\
{^LS}&&\rTo^{^Lj}&{^LG}\\
\uTo^{a}&&&\uTo^\varphi\\
W_F&&\rTo^=&W_F\\
\end{diagram}
where $$^Lj^\delta:=\delta\circ{^Lj}\circ\delta^{-1}$$ is also an
unramified $L$-embedding. In summary, we obtain that
$\left(S^\sigma,{^Lj^\delta},a^\sigma,\chi^\sigma\right)$ is a
regular depth-zero supercuspidal $L$-packet datum for the parameter
$\delta\circ\varphi$.

Each character $\mu\in Z_{\varphi}^D$ induces a character
$\mu\circ\delta^{-1}$ in $Z_{\delta\circ\varphi}^D$. Now we discuss
the relation between $(S^\sigma)_{\mu\circ\delta^{-1}}$ and $S_\mu$,
and the relation between $(\chi^\sigma)_{\mu\circ\delta^{-1}}$ and
$\chi_\mu$. First we have the following commutative diagram
\begin{diagram}
Z^D_{\varphi}&&\rTo&X_*(S)_\Gamma\\
\dTo&&&\dTo\\
Z^D_{\delta\circ\varphi}&&\rTo& X_*(S^\sigma)_\Gamma\\
\end{diagram} where the top and the bottom maps are
$\mu\mapsto\bar{\lambda}_\mu$ in (\ref{equ. bijection
cocharacters}), and the right map is
$\bar{\lambda}\mapsto\overline{\sigma\circ\lambda}$. It is easy to
check that the 1-cocycles $z_\mu$ and $z_{\mu\circ\delta^{-1}}$ in
$Z^1(W_F,G(\bar{L}))$, and the inner twists $\xi_\mu$ and
$\xi_{\mu\circ\delta^{-1}}$, which are determined by $\mu$ and
$\mu\circ\delta^{-1}$ respectively, satisfy the relation
$$z_{\mu\circ\delta^{-1}}=\sigma\circ z_\mu,\quad
\xi_{\mu\circ\delta^{-1}}=\sigma\circ\xi_\mu\circ\sigma.$$ Thus the
admissible embedding $j_{\mu\circ\delta^{-1}}$ of $S^\sigma$
attached to $\mu\circ\delta^{-1}$ is
$$S^\sigma\incl
G\stackrel{{\xi_{\mu\circ\delta^{-1}}}}{\lra}(G_\mu)^\sigma.$$ Hence
we see that if $\mu$ is in $Z_{\varphi}^{D,\circ}$ then
$\mu\circ\delta^{-1}$ is in $Z_{\delta\circ\varphi}^{D,\circ}$ and
$$\mathsf{b}(\mu\circ\delta^{-1})=\mathsf{b}(\mu).$$ Moreover we
have
$$(S^\sigma)_{\mu\circ\delta^{-1}}=(S_\mu)^\sigma,\quad
(\chi^\sigma)_{\mu\circ\delta^{-1}}=(\chi_\mu)^\sigma.$$ Therefore,
\begin{equation}\label{equ. relation between representations}
\pi\left((S^\sigma)_{\mu\circ\delta^{-1}},(\chi^\sigma)_{\mu\circ\delta^{-1}}\right)=
\pi\left((S_\mu)^\sigma,(\chi_\mu)^\sigma\right)
\end{equation}
From now on, we simply denote $$S_\mu^\sigma=(S_\mu)^\sigma,\quad
\chi_\mu^\sigma=(\chi_\mu)^\sigma.$$ In summary, according to
(\ref{equ. relation between representations}), we have shown that
$$\Pi_{\delta\circ\varphi}^\circ
=\left\{\pi\left(S_\mu^\sigma,\chi_\mu^\sigma\right):\
\mu\in Z_{\varphi}^{D,\circ}\right\}.$$ Hence Proposition \ref{prop.
packet of galois conjugate} follows from the lemma below directly.
\end{proof}

\begin{lem}\label{lem. galois conjugate and induction}
For any $\mu\in Z_{\varphi}^{D,\circ}$, we have
$$\pi\left(S_\mu^\sigma,\chi_\mu^\sigma\right)\simeq\pi\left(S_\mu,\chi_\mu\right)^\sigma.$$
\end{lem}

\begin{proof}
For any open subgroup $K$ of $G_b(F)$ which is compact modulo the
center, and any representation $(\rho,V_\rho)$ of $K$, we have
\begin{equation}\label{equ. induction and galois conjugate}
\left(\ind_K^{G_b(F)}\rho\right)^\sigma\simeq
\ind_{K^\sigma}^{G_b(F)}\left(\rho^\sigma\right),\end{equation}
where $K^\sigma$ is the image of $K$ under $\sigma$, and
$\left(\rho^\sigma,V_{\rho^\sigma}\right)=(\rho\circ\sigma,V_\rho)$
is the representation of $K^\sigma$. The intertwining operator is
given by
$$i:f\longmapsto f\circ\sigma,$$ where $f:G_b(F)\ra V_\rho$ is
any function in the space of
$\left(\ind_K^{G_b(F)}\rho\right)^\sigma$. Let $G_b(F)_{x,0}$ be the
parahoric subgroup corresponding to $S_\mu$. Then we have
$$G_b(F)^\sigma_{x,0}=G_b(F)_{\sigma(x),0},$$  where $G_b(F)_{\sigma(x),0}$
is the parahoric subgroup corresponding to $S_\mu^\sigma$. Let
$\rho_\mu$ be the representation of $G_b(F)_{x,0}$ constructed via
$(S_\mu,\chi_\mu)$, and $\rho_{\mu,\sigma}$ the representation of
$G_b(F)_{\sigma(x),0}$ constructed via
$(S_\mu^\sigma,\chi_\mu^\sigma)$. It is straightforward that
$\rho_{\mu,\sigma}\simeq\rho^\sigma_\mu$, which implies the lemma by
(\ref{equ. induction and galois conjugate}).
\end{proof}

Since $\pi\left(S,\chi\right)$ is $(B,\psi)$-generic, the
representation $\pi\left(S^\sigma,\chi^\sigma\right)$ is
$(B,\psi^\sigma)$-generic. Hence it makes sense to  denote both of
the two bijective maps
$$Z^D_{\delta\circ\varphi}\lra\Pi_{\delta\circ\varphi}\quad
\textrm{and}\quad
Z^{D,\circ}_{\delta\circ\varphi}\lra\Pi^\circ_{\delta\circ\varphi}$$
by $\iota_{B,\psi^\sigma}$. The following proposition, which is
finer than Proposition \ref{prop. packet of galois conjugate}, is a
direct consequence of (\ref{equ. relation between representations})
and Lemma \ref{lem. galois conjugate and induction}.

\begin{prop}\label{prop. correspondence for galois conjugate}
For any $\mu\in Z_{\varphi}^{D,\circ}$, we have
$$\iota_{B,\psi}\left(\varphi,\mu\right)^\sigma
=\iota_{B,\psi^\sigma}\left(\delta\circ\varphi,\mu\circ\delta^{-1}\right).$$
\end{prop}

\section{Distinction}\label{sec. distinction}

\subsection{Finite fields}\label{subsec. finite fields}
In this subsection, we recall a result on the problem of distinction
of representations for groups over finite fields. Let $\mathsf{H}$
be a connected reductive group over a finite field $k$, and $k'$ a
quadratic field extension of $k'$. Denote
$\mathsf{G}=\R_{k'/k}\mathsf{H}$. Then the Frobenius map $\F$
corresponding to the $k$-structure of $\mathsf{H}$ defines an
involution $\theta$ on $\mathsf{G}$. Let $\mathsf{S}$ be an elliptic
maximal torus of $\mathsf{G}$ over $k$, and
$\chi:\mathsf{S}(k)\ra\BC^\times$ a regular character. Then, by
Deligne-Lusztig theory, there is an irreducible cuspidal
representation $\rho$ of $\mathsf{G}(k)$ associated to
$(\mathsf{S},\chi)$. The following lemma is a special case of
\cite[Lemma 2.2]{lu00}.
\begin{lem}\label{lem. finite field}
The space $\Hom_{\mathsf{H}(k)}(\rho,1)$ is nonzero if and only if
there exists $g\in\mathsf{G}(k)$ such that
$\mathsf{S}^g=g\mathsf{S}g^{-1}$ is an $\F$-stable torus, and
$$\chi^g|_{\mathsf{S}^g(k)^\theta}=1$$ where
$\chi^g=\chi\circ\Ad(g^{-1})$.
\end{lem}
The main theorem of \cite{lu00} assumes that the center $\mathsf{Z}$
of $\mathsf{G}$ is connected and $\mathsf{G}/\mathsf{Z}$ is simple.
For \cite[Lemma 2.2]{lu00}, this assumption is not needed.

\subsection{Proof of Theorem \ref{thm. theorem on
representation}}\label{subsec. proof}

Let $\varphi:W_F\ra{^LG}$ be a TRSELP, and $\Pi_{\varphi}$ the
corresponding enlarged $L$-packet. Fix a character $\mu\in
Z_{\varphi}^{D,\circ}$, and let $b=\mathsf{b}(\mu)$. Choose
$a\in\FB(H)_\bsc$ such that $\res(a)=b$. Only in this subsection, by
abuse of notation, we denote
$$\left(\pi,S,\chi,G,H\right)=\left(\pi\left(S_\mu,\chi_\mu\right)
,S_\mu,\chi_\mu,G_b,H_a\right)$$ for simplicity. The following
proposition, which is analogous to Lemma \ref{lem. finite field}, is
the key point for the proof of Theorem \ref{thm. theorem on
representation}.

\begin{prop}\label{prop. stable torus}
Suppose that $\pi$ is $H(F)$-distinguished. Then there exists $g\in
G(F)$ such that $S^g$ is $\theta$-stable and
$$\chi^g|_{S^g(F)^\theta}=1.$$
\end{prop}
\begin{proof}
Let $x=\mathcal{A}^\red(S,F)$, which a vertex in
$\mathcal{B}^\red(G,F)$. Write $K$ for $Z(F)G(F)_{x,0}$, and
$\wt{K}$ for $G(F)_x$ which is the stabilizer of $x$ for the action
of $G(F)$ on $\mathcal{B}^\red(G,F)$. It is known that $\wt{K}$ is
the normalizer of $K$ in $G(F)$ and $K$ is of finite index in
$\wt{K}$.

Since $\pi$ is $H(F)$-distinguished, applying Hakim-Murnaghan theory
(cf. \cite[Theorem 5.26]{hm08}), we can and do assume that
$\theta(x)=x$, which implies that $\wt{K}$, $K$ and $G(F)_{x,0}$ are
all $\theta$-stable. Therefore the vertex $x$ lies in
$\CB^\red(H,F)$. Denote by $\sH_x$ the canonical connected smooth
group scheme over $O_F$, which is an integral model of $H$, such
that $\sH_x(O_F)=H(F)_{x,0}$. Since $E$ is unramified over $F$, we
have $\sH_x(O_E)=G(F)_{x,0}$. Let $\mathsf{G}$ and $\mathsf{H}$ be
the corresponding connected reductive groups over $k_F$. Then
$\mathsf{G}=\R_{k_E/k_F}\mathsf{H}$. The involution $\theta$ acts on
$G(F)_{x,0}$, and induces an involution $\bar{\theta}$ on
$\mathsf{G}$. Moreover the involution $\bar{\theta}$ is the
Frobenius map $\F$ with respect to the $k_F$-structure of
$\mathsf{H}$.

By abuse of notation, we use $\rho$ to denote both of the
representations $\rho_\mu\otimes\chi_\mu$ of $K$ and $\rho_\mu$ of
$G(F)_{x,0}$. Let $\wt{\rho}=\ind_K^{\wt{K}}\rho$, which is an
irreducible representation of $\wt{K}$. According to \cite[Theorem
5.26]{hm08}, we have
$$\Hom_{H(F)}(\pi,1)\simeq\bigoplus_{g_i}\Hom_{\wt{K}\cap
G(F)^{\theta_i}}(\wt{\rho},1)=\bigoplus_{g_i}\Hom_{\wt{K}^{\theta_i}}(\wt{\rho},1),$$
where $g_i$ runs over representatives of the double cosets
$\wt{K}\bs G(F)/H(F)$ such that $\tau(g_i)\in \wt{K}$ (which does
not depend on the choice of the representatives), and $\theta_i$'s
are the $g_i$-twisted involutions defined by
$\theta_i(y)=\tau(g_i)^{-1}\theta(y)\tau(g_i)$. Recall that
$\tau(g)$ is defined to be $g\theta(g)^{-1}$. Applying Mackey
theory, we have a finer decomposition
$$\Hom_{H(F)}(\pi,1)\simeq\bigoplus_{g_i}\Hom_{K^{\theta_i}}(\rho,1)$$
where each $g_i\in K\bs G(F)/H(F)$ satisfies $\tau(g_i)\in\wt{K}$.

For each $i$, write $x_i=\tau(g_i)$. Then
$\theta_i=\Ad(x_i^{-1})\circ\theta$. Denote by $\bar{\theta}_i$ the
involution on $\mathsf{G}(k_F)=\mathsf{H}(k_E)$ induced by
$\theta_i$. Now we want to show that $\bar{\theta}_i$ is actually a
Frobenius map on $\mathsf{H}(\bar{k}_F)$. For any
$\bar{y}\in\mathsf{H}(\bar{k}_F)$ let $y\in\sH_x(O_{F^u})$ be a lift
of $y$. We view $\theta$ as the automorphism on $\sH_x(F^u)$ induced
by the Frobenius map in $\Gal(F^u/F)$. Thus we have
$$\bar{\theta}_i^2(\bar{y})=\overline{x_i^{-1}\theta\left(x_i^{-1}\theta(y)x_i\right)
x_i}=\overline{\theta^2(y)}=\F^2(\bar{y}).$$ Therefore, for each
$i$, $\bar{\theta}_i$ is a Frobenius map. Denote $\bar{\theta}_i$ by
$\F_i$ for simplicity. Then
$\mathsf{G}(k_F)=\mathsf{H}({\bar{k}_F})^{\F_i^2}$. Let
$\mathsf{H}_i$ be the connected reductive group over $k_F$ such that
$\mathsf{H}_i(k_F)=\mathsf{H}({\bar{k}_F})^{\F_i}$.

For each $i$, if $\Hom_{K^{\theta_i}}(\rho,1)$ is nonzero then
$\Hom_{G(F)_{x,0}^{\theta_i}}(\rho,1)$ is nonzero. Note that
$$\Hom_{G(F)_{x,0}^{\theta_i}}(\rho,1)=\Hom_{\mathsf{H}_i(k_F)}(\bar{\rho},1).$$
Since $\Hom_{H(F)}(\pi,1)$ is nonzero, there exits an index $i$ such
that $\Hom_{\mathsf{H}_i(k_F)}(\bar{\rho},1)$ is nonzero. As before,
let $\mathsf{S}$ be the elliptic maximal torus in $\mathsf{G}$
corresponding to $S$, and $\bar{\chi}$ the character of
$\mathsf{S}(k_F)$ induced by $\chi$. According to Lemma \ref{lem.
finite field}, there exists an $\F_i$-stable torus $\mathsf{S}'$ in
$\mathsf{G}$ and a character $\bar{\chi}'$ of $\mathsf{S}'(k_F)$
satisfying
$$\bar{\chi}'|_{\mathsf{S'}(k_F)^{\F_i}}=1$$ such that
$(\mathsf{S},\bar{\chi})$ and $(\mathsf{S}',\bar{\chi}')$ are
$\mathsf{G}(k_F)$-conjugate.

By \cite[Lemma A.2]{hj12}, there exists a $\theta_i$-stable elliptic
unramified maximal torus $S'$ of $G$ such that $x=\CA^\red(S',F)$
and the image of $S'(F^u)\cap G(F^u)_{x,0}$ in
$\mathsf{G}(\bar{k}_F)$ is $\mathsf{S}'(\bar{k}_F)$. Since
$\mathsf{S}$ and $\mathsf{S}'$ are $\mathsf{G}(k_F)$-conjugate, $S$
and $S'$ are $G(F)_{x,0}$-conjugate. Choose $g_0\in G(F)_{x,0}$ such
that $S'=S^{g_0}$. Let $\chi'=\chi^{g_0}$. It is easy to check that
$\overline{\chi'}=\bar{\chi}'$.  Therefore
$\overline{\chi'}|_{\mathsf{S'}(k_F)^{\F_i}}=1$ and thus
$$\chi'|_{S'(O_F)^{\theta_i}}=1.$$ Let $Z'$ be the center of $H$. The
central character of $\pi$ is $\chi|_{Z(F)}$, which is equal to
$\chi'|_{Z(F)}$. Since $\pi$ is $H(F)$-distinguished, we have
$$\chi'|_{Z'(F)}=1.$$ Since $S'$ is elliptic
and unramified, we have $S'(F)^{\theta_i}=Z'(F)S'(O_F)^{\theta_i}$.
Therefore $$\chi'|_{S'(F)^{\theta_i}}=1.$$ Recall that
$\theta_i=\Ad(x_i^{-1})\circ\theta$ with $x_i=\tau(g_i)$. Let
$g'=\theta(g_i)^{-1}$. It is obvious that $S''=S'^{g'}$ is
$\theta$-stable and $S''(F)^\theta=g'S'(F)^{\theta_i}g'^{-1}$. Let
$\chi''=\chi'^{g'}$. Then $$\chi''|_{S''(F)^\theta}=1.$$ In summary,
if we set $g=g'g_0$ then $g$ satisfies the desired conditions in
Proposition \ref{prop. stable torus}.
\end{proof}

\begin{prop}\label{prop. distinction for trivial}
If $\pi$ is $H(F)$-distinguished, we have
$$\pi^\vee\simeq\pi^\sigma.$$
\end{prop}

\begin{proof}
According to \cite[Lemma 3.1.1]{ka14}, the equivalence class of
$\pi$ depends only on the $G(F)$-conjugate class of $(S,\chi)$. By
Proposition \ref{prop. stable torus}, we can and do assume that $S$
is $\theta$-stable and $\chi|_{T(F)}=1$, where $T=S^\theta$ is a
maximal torus of $H$. The condition $\chi|_{T(F)}=1$ implies that
$\chi^\sigma=\chi^{-1}$. Combining (\ref{equ. induction and
contragredient}) and Lemma \ref{lem. galois conjugate and
induction}, we obtain
$$\pi\left(S,\chi\right)^\sigma\simeq\pi\left(S^\sigma,\chi^\sigma\right)
=\pi\left(S,\chi^{-1}\right) \simeq\pi\left(S,\chi\right)^\vee.$$
Therefore, $\pi^\sigma\simeq\pi^\vee$.
\end{proof}

Now we have proved Theorem \ref{thm. theorem on representation} when
$\omega_H$ is trivial. We will construct an inflation $\omega_G$ of
$\omega_H$, which is a character of $G(F)$, and reduce Theorem
\ref{thm. theorem on representation} to Proposition \ref{prop.
distinction for trivial}.

First let us recall the construction of $\omega_H$. Let $\alpha$ be
the element in $H^1(W_F,\BZ/2\BZ)$ associated with the quadratic
extension $E$ of $F$. By choosing a regular unipotent element in
$\wh{H_\ad}$, we obtain a morphism $\SL_2(\BC)\ra\wh{H_\ad}$ and
thus a morphism $\beta:\BZ/2\BZ\ra Z(\wh{H_\ad})$ by restriction to
the center. It is explained in \cite[\S7]{pr01} that all characters
of $H(F)$ come from $H^1\left(W_F,Z(\wh{H_\ad})\right)$ and the
character $\omega_H$ is the one attached to $\beta\circ\alpha$.

Let $E'$ be the unique unramified extension of $E$ in $\bar{F}$ and
$\alpha'$ the element in $H^1(W_E,\BZ/2\BZ)$ associated with $E'$.
Similarly as the construction of $\omega_H$, we get a character
$\omega_G$ of $G(F)=H(E)$ associated to $\beta\circ\alpha'$. Note
that $\alpha$ factors through $W_F/I_F\simeq\BZ$ with kernel
$W_E/I_F$, and $\alpha'$ factors through $W_E/I_F\simeq2\BZ$, as a
subgroup of $W_F/I_F$, with kernel $W_{E'}/I_F$. Hence we have the
following commutative diagram
\begin{diagram}  W_F/I_F&&\rTo^{\alpha}
&\BZ/2\BZ\\
\dTo&&&\dTo\\
W_E/I_F&&\rTo^{\alpha'}& \BZ/2\BZ\\
\end{diagram}
where the left map is $t\mapsto 2t$. Therefore $\omega_G$ is an
inflation of $\omega_H$. Let $\omega_G^\sigma$ be the Galois
conjugate of $\omega_G$ for $\sigma\in\Gal(E/F)$.  Note that
$\Gal(E/F)\simeq W_F/W_E$ has a natural action on $W_E/I_F$ and
induces an action on $H^1(W_E/I_F,\BZ/2\BZ)$. Then $\omega_G^\sigma$
is the character of $G(F)$ associated with
$\beta\circ\sigma(\alpha')$. On the other hand, since $W_F/I_F$ is
abelian, $\sigma(\alpha')=\alpha'$. In summary, we have the
following lemma.

\begin{lem}\label{lem. inflate character}
The quadratic character $\omega_G$ is an inflation of $\omega_H$ and
satisfies $\omega_G^\sigma=\omega_G$.
\end{lem}

\begin{proof}[Proof of Theorem \ref{thm. theorem on representation}]
Suppose that $\pi$ is $\omega_H$-distinguished. Since $\omega_G$ is
an inflation of $\omega_H$, we have
$$\Hom_{H(F)}(\pi\otimes\omega_G,1)=\Hom_{H(F)}(\pi,\omega_G)
=\Hom_{H(F)}(\pi,\omega_H)\neq0.$$ In other words, we have
$\pi\otimes\omega_G$ is $H(F)$-distinguished. Since
$\pi\otimes\omega_G$ is also of depth zero, according to Proposition
\ref{prop. distinction for trivial}, we have
$$\pi^\vee\otimes\omega_G^{-1}\simeq
\left(\pi\otimes\omega_G\right)^\vee\simeq
\left(\pi\otimes\omega_G\right)^\sigma\simeq
\pi^\sigma\otimes\omega_G^\sigma.$$ By Lemma \ref{lem. inflate
character}, we have
$$\pi^\vee\simeq\pi^\sigma\otimes\omega^\sigma_G\omega_G=\pi^\sigma.$$

\end{proof}

\subsection{Consequences}\label{subsec. consequences}
Let $\varphi$ be a TRSELP for $G$ as before. Based on the
description of $\Pi^\vee_{\varphi}$ and
$\Pi_{\varphi}^{\circ,\sigma}$, we can investigate the problem of
distinction in terms of $L$-parameters. Recall that the hyperspecial
vertex $o$ and the Whittaker datum $(B,\psi)$ are chosen to be
$\sigma$-stable. Let $B'$ be the Borel subgroup of $H$ such that
$B=\R_{E/F}B'$, and $U'$ the unipotent radical of $B'$. Then
$U=\R_{E/F}U'$. Moreover we choose $\psi$ in the way that its
restriction to $U'(F)$ is trivial. Hence $\psi^\sigma=\psi^{-1}$.
From now on, for $\mu\in Z_{\varphi}^{D,\circ}$, we call
$\pi\left(S_\mu,\chi_\mu\right)$ {\em distinguished} for short if it
is $\omega_{H_a}$-distinguished for some $a\in\FB(H)_\bsc$ such that
$\res(a)=\mathsf{b}(\mu)$. If $\varphi'$ is another Langlands
parameter for $G$, we write $\varphi\sim\varphi'$ if these two
parameters are $\wh{G}$-conjugate. For $\mu\in Z_\varphi^D$ and
$\mu'\in Z_{\varphi'}^D$, we write
$(\varphi,\mu)\sim(\varphi',\mu')$ if there exists $\wh{g}\in\wh{G}$
such that $\varphi'=\Ad(\wh{g})\circ\varphi$ and
$\mu'=\Ad(\wh{g})\circ\mu$.

\begin{thm}\label{thm. distinction and parameters}
Suppose that $\pi\left(S_\mu,\chi_\mu\right)$ is distinguished. Then
we have
$$(C\circ\varphi,\mu)\sim(\delta\circ\varphi,\mu\circ\delta^{-1}).$$
\end{thm}
\begin{proof}
First, by (\ref{equ. correspondence for contragredient}), we have
$$\pi\left(S_\mu,\chi_\mu\right)^\vee
=\iota_{B,\psi^{-1}}({C\circ\varphi},\mu).$$ On the other hand, by
Proposition \ref{prop. correspondence for galois conjugate}, we have
$$\pi\left(S_\mu,\chi_\mu\right)^\sigma
=\iota_{B,\psi^\sigma}(\delta\circ\varphi,\mu\circ\delta^{-1})
=\iota_{B,\psi^{-1}}(\delta\circ\varphi,\mu\circ\delta^{-1}).$$ Then
the conclusion follows from Theorem \ref{thm. theorem on
representation}.
\end{proof}

\paragraph{The group $H^\op$.}
Now we interpret Theorem \ref{thm. distinction and parameters} in
the language of functoriality. By twisting the Galois structure
$\Gal(\bar{F}/F)\ra\Out(H(\bar{F}))$ of $H$ via sending
$\sigma\in\Gal(E/F)$ to a fixed Chevalley involution $C$ (over $F$)
in $\Out(H(\bar{F}))$, we obtain a quasi-split reductive group
$H^\op$ over $F$, which is isomorphic to $H$ over $E$. We can
identify the $L$-group $^LH^\op$ of $H^\op$ as the subgroup
$$\left\{\left(h,C(h)\right): h\in\wh{H}\right\}\rtimes W_F$$ of
$^LG=\left(\wh{H}\times\wh{H}\right)\rtimes W_F$ up to
$\wh{G}$-conjugacy, where $C$ is a Chevalley involution on $\wh{H}$.

The following corollary, which is a direct consequence of Theorem
\ref{thm. distinction and parameters}, shows that a regular
depth-zero supercuspidal $L$-packet which contains a distinguished
representation should be a functorial lift from $H^\op$ via the base
change map $^LH^\op\ra{^LG}$.

\begin{cor}\label{cor. distinction and packets}
If $\pi\left(S_\mu,\chi_\mu\right)$ is distinguished for some
$\mu\in Z_{\varphi}^{D,\circ}$, we have
\begin{enumerate}
\item $\Pi_{C\circ\varphi}=\Pi_{\delta\circ\varphi},$
\item The parameter $\varphi$ factors through $^L{H^\op}$.
\end{enumerate}
\end{cor}

\begin{proof}
The first assertion is due to Theorem \ref{thm. distinction and
parameters}. For the second assertion, write the projection
$\varphi_0$ of $\varphi$ to $\wh{G}$ as $(\varphi_1,\varphi_2)$ so
that
$$\varphi_0(w)=\left(\varphi_1(w),\varphi_2(w)\right)$$ for
$w\in W_F$. Then the condition $C\circ\varphi\sim\delta\circ\varphi$
implies that $\varphi_2\sim C\circ\varphi_1$, that is, $\varphi$
factors through $^LH^\op$.

\end{proof}

\s{\small Chong Zhang\\
School of Mathematical Sciences, Beijing Normal University,\\
Beijing 100875, P. R. China.\\
E-mail address: \texttt{zhangchong@bnu.edu.cn}}

\end{document}